\documentclass[a4paper,10pt]{amsart}

\usepackage{graphicx}
\usepackage{geometry}
\usepackage{amsthm}
\usepackage{amssymb}
\usepackage{amsmath}
\usepackage{accents}
\usepackage{hyperref}
\usepackage{xcolor}
\usepackage{enumerate}
\usepackage{listings}
\usepackage{complexity}
\usepackage{blkarray}
\usepackage{braket}
\usepackage{lmodern}
\usepackage[english]{babel}

\usepackage[font=small,labelfont=bf]{caption}
\usepackage[font=small,labelfont=normalfont,labelformat=simple]{subcaption}
\graphicspath{{figs/}}

\newtheorem{theorem}{Theorem}
\newtheorem{definition}{Definition}
\newtheorem{lemma}[theorem]{Lemma}
\newtheorem{corollary}[theorem]{Corollary}

\newtheorem{observation}{Observation}
\newtheorem{proposition}{Proposition}

\newtheorem{problem}{Problem}

\author
{
Raphael Steiner 
}
\thanks{Department of Computer Science, Institute of Theoretical Computer Science, ETH Z\"{u}rich, Switzerland,  \texttt{raphaelmario.steiner@inf.ethz.ch}. The author was supported by an
ETH Zurich Postdoctoral Fellowship.}

\date{\today}

\title{A logarithmic bound for simultaneous embeddings of planar graphs}\thanks{A short version of this paper appears in the proceedings of the 31st International Symposium on Graph Drawing and Network Visualization. The present paper contains not only several added details, but also completely new material, including Sections~\ref{sect:smalln} and~\ref{sect:explicit}.}

\begin{document}
\maketitle

\begin{abstract}
A set $\mathcal{G}$ of planar graphs on the same number $n$ of vertices is called \emph{simultaneously embeddable} if there exists a set $P$ of $n$ points in the plane such that every graph $G \in \mathcal{G}$ admits a (crossing-free) straight-line embedding with vertices placed at points of $P$. A \emph{conflict collection} is a set of planar graphs of the same order with no simultaneous embedding. A well-known open problem from 2007 posed by Brass, Cenek, Duncan,
Efrat, Erten, Ismailescu, Kobourov, Lubiw and Mitchell, asks whether there exists a conflict collection of size $2$. While this remains widely open, we give a short proof that for sufficiently large $n$ there exists a conflict collection consisting of at most $(3+o(1))\log_2(n)$ planar graphs on $n$~vertices. This constitutes a double-exponential improvement over the previously best known bound of $O(n\cdot 4^{n/11})$ for the same problem by Goenka, Semnani and Yip.

Using our method we also provide a computer-free proof that there exists a conflict collection of size $30$, improving upon the previously smallest known conflict collection of size $49$ which was found using heavy computer assistance.

While the construction by Goenka et al.~was explicit, our construction of a conflict collection of size $O(\log n)$ is based on the probabilistic method and is thus only implicit. Motivated by this, for every large enough $n$ we give a different, fully explicit construction of a collection of less than $n^6$ planar $n$-vertex graphs with no simultaneous embedding.
\end{abstract}

\section{Introduction}

Given a planar graph $G$, a \emph{straight-line embedding} of $G$ in the plane is an injective mapping $\pi$ from the vertex-set $V$ of $G$ to $\mathbb{R}^2$ such that adding in all the straight-line segments with endpoints $\pi(u)$ and $\pi(v)$ for every edge $uv$ in $G$ yields a crossing-free drawing of $G$ in the plane. Given a planar graph $G$ and a point set $P\subseteq \mathbb{R}^2$ in the plane, we say that $G$ \emph{has a straight-line embedding on $P$} or equivalently that it \emph{embeds straight-line on $P$} if there exists a straight-line embedding of $G$ in the plane in which every vertex of $G$ is mapped to a distinct point of~$P$. 
If $G$ is a \emph{labelled planar graph} with vertices numbered as $\{v_1,\ldots,v_n\}$, and if $P$ is a \emph{labelled point set}, that is, its elements are numbered as $P=\{p_1,\ldots,p_n\}$, then we say that $G$ has a \emph{label-preserving straight-line embedding on $P$} if the bijection $\pi:V(G)\rightarrow P$, $\pi(v_i):=p_i$, forms a straight-line embedding of $G$. 

The study of straight-line embeddings of planar graphs is a classical area in graph drawing. For instance, one of the most fundamental results on this topic, the F\'{a}ry-Wagner-Theorem~\cite{fary}, states that every planar graph $G$ admits a straight-line embedding in the plane on \emph{some} point set. However, it is not true that a planar graph on $n$ vertices can be embedded on any given point set of size $n$. In fact, for a fixed planar graph $G$ on $n$ vertices, only a small fraction of all potential $n$-point sets may allow a straight-line embedding of $G$. Thus, many interesting questions in graph drawing arise from considering the embeddability of (restricted classes of) planar graphs on (restricted types of) point sets. One of the biggest branches of research in this direction concerns \emph{simultaneous embeddings} of sets of planar graphs, we refer to~\cite{blaesius} for a survey on this topic. Given a set $\mathcal{G}$ of planar graphs and a point set $P$ in the plane, we say that $\mathcal{G}$ is \emph{simultaneously embeddable on $P$} if every member $G \in \mathcal{G}$ admits a straight-line embedding on $P$. If $n \in \mathbb{N}$ and $\mathcal{G}$ is a set of planar graphs, each on $n$ vertices, we say that $\mathcal{G}$ is \emph{simultaneously embeddable} (without mapping) if there exists a point set $P \subseteq \mathbb{R}^2$ of size $n$ such that $\mathcal{G}$ is simultaneously embeddable on $P$.

There are two major open problems in geometric graph theory related to the notions introduced above. The first, called the \emph{universal set problem}, asks to find the asymptotics of the function $f(n)$, defined as the smallest size of a point set $P$ in the plane such that the set of all $n$-vertex planar graphs is simultaneously embeddable on $P$ (such a set is called \emph{$n$-universal}). Currently there is still a large gap in our understanding of this problem, with the best asymptotic estimates being $f(n) \le \frac{1}{4}n^2+O(n)$ by Bannister et al.~\cite{bannister} and $f(n) \ge (1.293-o(1))n$ by Scheucher et al.~\cite{scheuch}.
%In another direction, 
Brass et al.~\cite{brass} initiated a systematic study of the simultaneous embeddability of small sets $\mathcal{G}$ of planar graphs. In particular, they raised the following intriguing open problem, which remains unsolved. 
\begin{problem}[cf.~\cite{brass}]\label{problem:hard}
Is there a set $\mathcal{G}=\{G_1,G_2\}$ consisting of two planar graphs of the same order such that $\mathcal{G}$ is not simultaneously embeddable?
\end{problem}
Following the terminology of~\cite{cardinal,expo,scheuch}, let us call a set $\mathcal{G}$ of planar graphs, all of the same order $n \in \mathbb{N}$, a \emph{conflict collection} if $\mathcal{G}$ is not simultaneously embeddable. Addressing small values of $n$, Cardinal et al.~\cite{cardinal} proved that for $n \le 10$, there exists no conflict collection consisting of $n$-vertex planar graphs. In contrast, they showed that for every $n \ge 15$ a conflict collection \emph{does} exist. 
Motivated by Problem~\ref{problem:hard}, it is natural to study the value~$\sigma(n)$, defined as the smallest size of a conflict collection of $n$-vertex planar graphs (if such a collection exists). By the F\'{a}ry-Wagner-Theorem, we have $\sigma(n) \ge 2$ for every $n$, and Problem~\ref{problem:hard} is equivalent to the question whether there exists some $n$ such that $\sigma(n)=2$. Approaching this question, Cardinal et al.~\cite{cardinal} constructed a relatively small conflict collection on $35$-vertex graphs, proving that $\sigma(35) \le 7393$. A significantly smaller conflict collection consisting of $11$-vertex graphs was found by Scheucher et al.~\cite{scheuch}, showing that $\sigma(11) \le 49$. Regarding the general asymptotic bounds on the function $\sigma(n)$, it was recently proved by Goenka et al.~\cite{expo} that $\sigma(n) \le O(n\cdot 4^{n/11})<O(1.135^n)$, by an explicit general construction of a conflict collection on $n$-vertex planar graphs. While this bound is exponential in $n$, in this paper we give a short probabilistic proof that $\sigma(n)=O(\log n)$, and thus for large enough $n$ much smaller conflict collections of $n$-vertex graphs of only logarithmic size in~$n$ exist.

\begin{theorem}\label{thm:main}
It holds that $\sigma(n)\le (3+o(1))\log_2(n)$.
\end{theorem} 

Using the same technique, but with a more careful analysis, we then obtain the following upper bounds, which improve upon the benchmark of $49$ for the size of the previously smallest known conflict collection of planar graphs~\cite{scheuch}. In contrast to the heavily computer-assisted proof of the bound $49$ in~\cite{scheuch}, our proof in this paper is computer-free, elementary and self-contained.

\begin{theorem}\label{thm:main2}
For every $n \in \{107,108,\ldots,193\}$, we have $\sigma(n)\le 30$. In particular, there exists a conflict collection consisting of $30$ planar graphs.
\end{theorem}

The proofs of Theorem~1 and~2 rely on the probabilistic method and thus unfortunately do not provide explicit constructions of the asserted conflict collections. Motivated by this, we also present a different, fully explicit construction of a conflict collection of less than $n^6$ planar $n$-vertex graphs for every large enough $n$. This still improves significantly over the explicit construction of size $(21n+552)4^{(n+37)/11}$ given by Goenka et al.~\cite{expo}, reducing the size of the constructed conflict collection from exponential to polynomial.

\begin{theorem}\label{thm:main3}
    For every $n \ge 7!=5040$, there exists an explicit construction of a conflict collection consisting of $n^{\underline{6}}+1=n(n-1)\cdots(n-5)+1$ planar $n$-vertex graphs. 
\end{theorem}

\paragraph{\textbf{Overview.}} The rest of this paper is structured as follows. In Section~\ref{sect:ordertypes}, we prepare the proofs of Theorems~\ref{thm:main} and~\ref{thm:main2} by introducing the necessary background and collecting some important facts regarding order types and straight-line embeddability. 
In Section~\ref{sect:mainproof} we then prove a general upper bound on $\sigma(n)$ in terms of the number $t_s(n,2)$ of simple order types of $n$ points (Theorem~\ref{thm:aux}). We then deduce Theorem~\ref{thm:main} by combining Theorem~\ref{thm:aux} with an asymptotic upper bound on the number of order types on $n$ points due to Alon~\cite{alon}. 
In Section~\ref{sect:smalln}, we use a theorem by Warren~\cite{warren} on the number of components of the set of common non-zeros of a system of multivariate polynomials to obtain upper bounds on $t_s(n,2)$ that are better than the bounds coming from Alon's paper for moderately small values of $n$. These improved estimates are significant when proving Theorem~\ref{thm:main2} for $n \in \{107,\ldots,193\}$. 
Finally, in Section~\ref{sect:explicit} we give the self-contained proof of Theorem~\ref{thm:main3}.

\section{Order types and straight-line embeddings}\label{sect:ordertypes}
Usually when working on a geometric problem relating to point sets that has a combinatorial flavour, small perturbations of the point set at hand do not change the behavior of the problem. In fact, often times one can reduce the infinite set of potential point sets $P$ in question to a finite set of ``types'', such that all point sets of the same type have identical behavior for the considered problem. This is also the case here, if one groups the point sets consisting of $n$ points in the plane according to their so-called \emph{order type}.

In the following definition, formally the \emph{orientation} of a triple $abc$ of three distinct points $a=(a_1,a_2)$, $b=(b_1,b_2)$, $c=(c_1,c_2)\in \mathbb{R}^2$ in the plane is defined as the sign of the determinant of the matrix 
$$\begin{pmatrix}
1 & 1 & 1\\
a_1 & b_1 & c_1\\
a_2 & b_2 & c_2
\end{pmatrix}.$$
It can be seen quite easily that the orientation of $abc$ is $1$ if the triangle $abc$ is oriented counterclockwise, it is $0$ if $a,b,c$ are collinear and it is $-1$ if the orientation of $abc$ is clockwise.

\begin{definition}[cf.~\cite{aichholzer,alon,cardinal,goaoc,scheuch}]
\begin{itemize}
\noindent
    \item Given two finite point sets $P$ and $Q$ in the plane, we say that $P$ and $Q$ are \emph{combinatorially equivalent} if there exists a bijection $f:P\rightarrow Q$ with the following property: 
For every ordered triple $abc$ consisting of three distinct points of $P$, the \emph{orientation} of the triple $abc$ is the same as the orientation of the triple $f(a)f(b)f(c)$. 
\item Given two \emph{labelled} point sets $P=\{p_1,\ldots,p_n\}$ and $Q=\{q_1,\ldots,q_n\}$ in the plane, we say that they are \emph{isomorphic} if for every choice of $1\le i,j,k \le n$ the orientations of the point triples $p_ip_jp_k$ and $q_iq_jq_k$ are identical.
\end{itemize}
\end{definition} 

It is readily verified that combinatorial equivalence and isomorphy form equivalence relations on the sets of unlabelled and labelled $n$-point sets in the plane, respectively. Accordingly, one can partition the set of $n$-point sets into equivalence classes w.r.t.~combinatorial equivalence, which are called \emph{order types}. Similarly, one can partition the set of labelled $n$-point sets into equivalence classes w.r.t.~isomorphy, and these are called \emph{labelled order types}. Note that by definition, two unlabelled point sets are combinatorially equivalent if and only if they admit labellings which are isomorphic.
The following observations, which are folklore and are explained, for instance, in~\cite{aichholzer,cardinal,scheuch}, show that (a) with regards to straight-line embeddability of planar graphs, it does not make a difference which among several combinatorially equivalent point sets we consider and (b) small perturbations of a point set do not change embeddability. 
\begin{observation}[cf.~\cite{aichholzer,cardinal,scheuch}]\label{embeddingequ}
\noindent
\begin{itemize}
    \item[(i)] If two point sets $P$ and 
    $Q$ in the plane are combinatorially equivalent, then for every planar graph $G$ it holds that $G$ embeds straight-line on $P$ if and only if it embeds straight-line on~$Q$. 
    
\item[(ii)] If $G$ is a labelled planar graph with $V(G)=\{v_1,\ldots,v_n\}$ and $P=\{p_1,\ldots,p_n\}$, $Q=\{q_1,\ldots,q_n\}$ are isomorphic labelled point sets, then $G$ admits a label-preserving straight-line embedding on $P$ if and only if it admits a label-preserving straight-line embedding on~$Q$. 

\item[(iii)] Let $G$ be planar, and let $\pi:V(G)\rightarrow P$ be a straight-line embedding of $G$ on a point set $P$. Then there exists $\varepsilon>0$ such that the following holds. For every choice of points $(x_p)_{p \in P}$ such that $\lVert x_p-p\rVert_2 <\varepsilon$ for every $p \in P$, the mapping $\pi':V(G)\rightarrow \{x_p|p \in P\}$ defined by $\pi'(v):=x_{\pi(v)}$ for every $v \in V(G)$ is a straight-line embedding of $G$.
\end{itemize}
\end{observation}

 In the following and in the rest of the manuscript, we use $t(n,2)$ to denote the number of labelled order types of $n$ points in the plane, and we use $t_s(n,2)$ to denote the number of \emph{simple} labelled order types of order $n$, i.e. those corresponding to $n$-point sets in the plane that are in \emph{general position} (i.e., where no three points lie on a common line). This notation is taken from the paper~\cite{alon} by Alon.

 The following consequence of Observation~\ref{embeddingequ} will be useful later.
\begin{corollary}\label{fewpointsets}
    For every $n\in \mathbb{N}$ there exists a collection $\mathcal{P}_n$ consisting of labelled $n$-point sets in $\mathbb{R}^2$ such that $|\mathcal{P}_n|= t_s(n,2)$ and such that the following holds. 
    
    For every collection $\mathcal{G}=\{G_1,G_2,\ldots,G_k\}$ of labelled $n$-vertex planar graphs, if $\mathcal{G}$ is simultaneously embeddable, then there exists $P \in \mathcal{P}_n$ such that 
    \begin{itemize}
        \item $G_1$ admits a label-preserving straight-line embedding on $P$, and
        \item each of $G_2,\ldots,G_k$ embeds straight-line on $P$ (but not necessarily in a label-preserving fashion).
    \end{itemize}
\end{corollary}
\begin{proof}
    Define $\mathcal{P}_n$ by selecting for each of the $t_s(n,2)$ distinct simple labelled order types one representative element, and adding it to $\mathcal{P}_n$. Then $|\mathcal{P}_n|= t_s(n,2)$. Now, let any collection $\mathcal{G}=\{G_1,\ldots,G_k\}$ of labelled $n$-vertex planar graphs be given, and suppose that $\mathcal{G}$ is simultaneously embeddable. Let $Q$ be an $n$-point set such that each of $G_1,\ldots,G_k$ embeds straight-line on $Q$. By applying item (iii) of Observation~\ref{embeddingequ} for each of the $G_i, i=1,\ldots,k$, we obtain that there exist $\varepsilon_1,\ldots,\varepsilon_k>0$ such that perturbing $Q$ by a distance of less than $\varepsilon_i$ at each point preserves the straight-line embeddability of $G_i$. Hence, perturbing $Q$ by a distance of less than $\varepsilon:=\min\{\varepsilon_1,\ldots,\varepsilon_k\}$ at each point preserves simultaneous embeddability of $\{G_1,\ldots,G_k\}$. Therefore, possibly after applying a suitable perturbation, we may assume w.l.o.g. that $Q$ is in general position. Let $V(G_1)=\{v_1,\ldots,v_n\}$ be the labelling of the vertices of $G_1$ and let $\pi_1:V(G_1) \rightarrow Q$ be a straight-line embedding of $G_1$ on $Q$. Now consider the labelling $Q=\{q_1,\ldots,q_n\}$ of $Q$, defined by setting $q_i:=\pi_1(v_i)$ for $i=1,\ldots,n$. Then clearly, $G_1$ has a label-preserving straight-line embedding on $\{q_1,\ldots,q_n\}$. Let $P \in \mathcal{P}_n$ be such that $P$ and $\{q_1,\ldots,q_n\}$ are isomorphic labelled point sets. By Observation~\ref{embeddingequ}, (ii) we have that $G_1$ admits a label-preserving straight-line embedding on $P$. Since $Q$ and $P$ have isomorphic labellings, they are combinatorially equivalent. Thus, by Observation~\ref{embeddingequ}, (i) each of $G_2,\ldots,G_k$ embeds straight-line on $P$, as desired. This concludes the proof. 
\end{proof}
\section{Proof of Theorem~\ref{thm:main}}\label{sect:mainproof}
We prepare the proof by introducing a specific family of labelled planar graphs (which are stacked triangulations), that has already been studied in a similar context in~\cite{cardinal,expo,scheuch}. 
\begin{definition}[cf.~\cite{scheuch}, Definition~2 and~\cite{cardinal}, Section~3]
For every integer $n \ge 4$, we define a set $\mathcal{T}_n$ of labelled planar triangulations, each with vertex-set $V_n=\{v_1,\ldots,v_n\}$, inductively as follows:
\noindent
\begin{itemize}
    \item $\mathcal{T}_4$ consists only of the complete graph $K_4$ on vertex-set $V_4=\{v_1,v_2,v_3,v_4\}$.
    \item If $T$ is a labelled graph in $\mathcal{T}_{n-1}$ with $n \ge 5$, and $v_iv_jv_k$ is a facial triangle in $T$, then the labelled planar graph obtained from $T$ by adding the new vertex $v_n$ and connecting it to $v_i, v_j, v_k$ is a member of $\mathcal{T}_n$. 
\end{itemize}
\end{definition}
It is crucial to notice that in the above definition and in the following $\mathcal{T}_n$ is to be understood as a set of \emph{labelled planar graphs}, and distinct elements of the set are distinguished even if their underlying planar graphs are isomorphic. Concretely, it may happen that the same planar graph but with different labellings occurs several times in $\mathcal{T}_n$.  Regardless, these different labellings are considered as different members of $\mathcal{T}_n$. For our proof of Theorem~\ref{thm:main} we will need a few properties of the family $\mathcal{T}_n$, summarized in the following lemma.
\begin{lemma}[cf.~\cite{scheuch}, Lemma~1 and~\cite{cardinal}, Lemma~1,2]~\label{lemma:reconstruct}
\noindent
\begin{itemize}
    \item[(i)] For every $n \ge 4$, the family $\mathcal{T}_n$ contains exactly $2^{n-4}(n-3)!$ distinct labelled planar graphs. 
    \item[(ii)] Let $P$ be any set of $n$ points in the plane. Then for every bijection $\pi:\{v_1,\ldots,v_n\}\rightarrow P$ there exists at most one $T \in \mathcal{T}_n$ with the property that placing the vertex $v_i$ of $T$ at the point $\pi(v_i)$ of $P$, for every $i=1,\ldots,n$, defines a straight-line embedding of $T$. 
\end{itemize}
\end{lemma}

We need the following simple and direct consequence of the previous lemma. 

\begin{corollary}\label{lemma:probability}
    Let $n\ge 4$ and let $\mathbf{G}$ denote a random labelled planar graph drawn uniformly from the set $\mathcal{T}_n$. Then the following hold. 
    \begin{itemize}
        \item[(i)] If $P=\{p_1,\ldots,p_n\}$ is a labelled $n$-point set in the plane, then 
        $$\mathbb{P}(\mathbf{G}\text{ has label-preserving straight-line embedding on }P)\le \frac{1}{2^{n-4}(n-3)!}.$$
        \item[(ii)] If $P$ is an $n$-point set in the plane, then
        $$\mathbb{P}(\mathbf{G}\text{ embeds straight-line on }P)\le \frac{16n(n-1)(n-2)}{2^n}.$$
    \end{itemize}
\end{corollary}
\begin{proof}
Let us first prove $(i)$. Note that by definition a graph $T \in \mathcal{T}_n$ has a label-preserving straight-line embedding on $\{p_1,\ldots,p_n\}$ if and only if the mapping $\pi:\{v_1,\ldots,v_n\} \rightarrow P$, $\pi(v_i):=p_i$ forms a straight-line embedding of $T$. By item $(ii)$ of Lemma~\ref{lemma:reconstruct} there is at most one member of $\mathcal{T}_n$ with this property. Since $\mathbf{G}$ follows a uniform distribution on $\mathcal{T}_n$, this implies, using item $(i)$ of Lemma~\ref{lemma:reconstruct}, $$\mathbb{P}(\mathbf{G}\text{ has label-preserving straight-line embedding on }P) \le \frac{1}{|\mathcal{T}_n|}=\frac{1}{2^{n-4}(n-3)!},$$ as desired.
Let us now prove item $(ii)$. To do so, let $p_1,\ldots,p_n$ be some enumeration of the points in $P$. Now, note that any labelled planar graph $G$ embeds straight-line on $P$ if and only if it has a label-preserving straight-line embedding on the relabelled point-set $\{p_{\tau(1)},\ldots,p_{\tau(n)}\}$ for some permutation $\tau$ of $\{1,\ldots,n\}$. We therefore have, using a union bound over the choices of $\tau$,
    $$\mathbb{P}(\mathbf{G}\text{ embeds straight-line on }P)$$ $$\le \sum_{\tau \in S_n}{\mathbb{P}(\mathbf{G}\text{ has label-preserving straight-line embedding on }\{p_{\tau(1)},\ldots,p_{\tau(n)}\})}$$ $$\le \frac{n!}{2^{n-4}(n-3)!}=\frac{16n(n-1)(n-2)}{2^n},$$ as claimed. 
\end{proof}

In the rest of the section we proceed in two parts. First we show Theorem~\ref{thm:aux} below, which gives a general upper bound on $\sigma(n)$ in terms of $t_s(n,2)$. Secondly, we use an upper bound due to Alon~\cite{alon} on $t_s(n,2)$ to complete the proof of Theorem~\ref{thm:main}.

\begin{theorem}\label{thm:aux}
For every integer $n\ge 16$, we have $$\sigma(n) \le \left\lfloor \frac{\log_2(t_s(n,2))-(n-4)-\log_2((n-3)!)}{n-\log_2(16n(n-1)(n-2))}\right\rfloor +2.$$  
\end{theorem} 
\begin{proof}
Let $k=\left\lfloor \frac{\log_2(t_s(n,2))-(n-4)-\log_2((n-3)!)}{n-\log_2(16n(n-1)(n-2))}\right\rfloor +2$, and let $\mathbf{G}_1,\ldots,\mathbf{G}_k$ denote a sequence of $k$ random planar graphs sampled independently from the uniform distribution on the set $\mathcal{T}_n$.

Let $\mathcal{G}:=\{\mathbf{G}_1,\ldots,\mathbf{G}_k\}$ denote the resulting multi-set of $k$ planar graphs. In the following we work towards arguing that the probability that the randomly generated multi-set $\mathcal{G}$ of planar $n$-vertex graphs is simultaneously embeddable is strictly less than $1$. To bound the probability that $\mathcal{G}$ is simultaneously embeddable, we recall Corollary~\ref{fewpointsets} which gives us a collection $\mathcal{P}_n$ of at most $t_s(n,2)$ labelled $n$-point sets with the property that if $\mathcal{G}$ is simultaneously embeddable, then there exists $P \in \mathcal{P}_n$ such that $\mathbf{G}_1$ has a label-preserving straight-line embedding on $P$, and each of $\mathbf{G}_2,\ldots,\mathbf{G}_k$ embeds straight-line on $P$. 

For any fixed set $P \in \mathcal{P}_n$, we can use the independence of $\mathbf{G}_1,\ldots,\mathbf{G}_k$ and Corollary~\ref{lemma:probability} to compute that
$$\mathbb{P}\left(\{\mathbf{G}_1 \text{ has label-pres. str.-line emb. on }P\} \wedge \bigwedge_{i=2}^{k}\{\mathbf{G}_i \text{ emb. str.-line on }P\}\right)$$
$$=\mathbb{P}(\mathbf{G}_1 \text{ has label-pres. str.-line emb. on }P)\cdot\prod_{i=2}^{k}{\mathbb{P}(\mathbf{G}_i \text{ emb. str.-line on }P)}$$
$$\le \frac{1}{2^{n-4}(n-3)!}\cdot \left(\frac{16n(n-1)(n-2)}{2^n}\right)^{k-1}.$$

Using a union-bound over all possible choices of $P \in \mathcal{P}_n$ we conclude:
$$\mathbb{P}(\mathcal{G} \text{ is simultaneously embeddable})$$ $$\le \sum_{P \in \mathcal{P}_n}\mathbb{P}\left(\{\mathbf{G}_1 \text{ has label-pres. str.-line emb. on }P\} \wedge \bigwedge_{i=2}^{k}\{\mathbf{G}_i \text{ emb. str.-line on }P\}\right)$$
$$\le |\mathcal{P}_n|\cdot \frac{1}{2^{n-4}(n-3)!}\cdot \left(\frac{16n(n-1)(n-2)}{2^n}\right)^{k-1}$$ $$= t_s(n,2) \cdot \frac{1}{2^{n-4}(n-3)!}\cdot \left(\frac{16n(n-1)(n-2)}{2^n}\right)^{k-1}.$$ Note that $$k-1=\left\lfloor \frac{\log_2(t_s(n,2))-(n-4)-\log_2((n-3)!)}{n-\log_2(16n(n-1)(n-2))}\right\rfloor+1> \frac{\log_2(t_s(n,2))-(n-4)-\log_2((n-3)!)}{n-\log_2(16n(n-1)(n-2))},$$ which implies after rearranging (using $n-\log_2(16n(n-1)(n-2))>0$, since $n \ge 16$) that $$\log_2(t_s(n,2))-(n-4)-\log_2((n-3)!)-(k-1)(n-\log_2(16n(n-1)(n-2)))<0.$$ Using this, we obtain:
$$\mathbb{P}(\mathcal{G} \text{ is simultaneously embeddable})$$ $$\le t_s(n,2) \cdot \frac{1}{2^{n-4}(n-3)!}\cdot \left(\frac{16n(n-1)(n-2)}{2^n}\right)^{k-1}$$ $$= 2^{\log_2(t_s(n,2))-(n-4)-\log_2((n-3)!)-(k-1)(n-\log_2(16n(n-1)(n-2)))}<2^0=1.$$  The probabilistic method now implies that there exists at least one multi-set consisting of $k$ planar graphs on $n$ vertices that is not simultaneously embeddable. The corresponding subset of at most $k$ planar graphs, obtained by removing repeated elements, now witnesses that $\sigma(n)\le k=\left\lfloor \frac{\log_2(t_s(n,2))-(n-4)-\log_2((n-3)!)}{n-\log_2(16n(n-1)(n-2))}\right\rfloor +2$, as desired.
\end{proof}

We are now ready to give the proof of Theorem~\ref{thm:main}. For this purpose, we use a result by Alon~\cite{alon} who provided asymptotically precise estimates for $t_s(n,2)$ in 1986. In contrast, the precise asymptotics of the number of unlabelled order types remains currently unknown, see the discussion in~\cite{goaoc}.

\begin{theorem}[cf.~\cite{alon}, Corollary 4.2]\label{thm:alon}
It holds that 
$$t_s(n,2)=n^{(4+o(1))n}.$$
\end{theorem}

Theorem~\ref{thm:main} can now be derived as an immediate consequence.

\begin{proof}[Proof of Theorem~\ref{thm:main}]
Plugging in the asymptotic formula $t_s(n,2)=n^{(4+o(1))n}$ from Theorem~\ref{thm:alon} into Theorem~\ref{thm:aux}, we obtain for $n \ge 16$, using the estimate $n!=n^{(1-o(1))n}$:
$$\sigma(n)\le  \frac{\log_2(t_s(n,2))-(n-4)-\log_2((n-3)!)}{n-\log_2(16n(n-1)(n-2))} +2$$
$$=\frac{(4+o(1))n\log_2(n)-(1-o(1))n-(1-o(1))n\log_2(n)}{(1-o(1))n}+2$$
$$=(3+o(1))\log_2(n)+1=(3+o(1))\log_2(n),$$
as claimed by the theorem.
\end{proof}

\section{Better bounds on the number of order types for small $n$}\label{sect:smalln}

In this section, we prove Theorem~\ref{thm:main2}, which provides better bounds on $\sigma(n)$ for moderately small values of $n$. To do so, we would like to again employ Theorem~\ref{thm:aux}. This time however, instead of the asymptotic estimate by Alon in Theorem~\ref{thm:alon}, a concrete bound is required. Looking into the proof of Theorem~\ref{thm:alon} in~\cite{alon}, it provides an explicit upper bound on $t_s(n,2)$ of the form:
$$t_s(n,2) \le \left(8\left(\frac{en}{2}\right)^2\ln(n/2))\right)^{2n+\ln(n/2)}.$$
We noticed that for small values of $n$, this bound can be considerably improved. Instead of using Alon's result directly, we make use of a theorem of Warren regarding the number of sign-patterns of a set of multivariate polynomials, to get a slightly more precise (but asymptotically of course equivalent) bound for $t_s(n,2)$. It turns out that this yields a considerable improvement on the resulting bound for $\sigma(n)$ when $n$ is moderately small. 

The following is the relevant theorem of Warren mentioned above.
\begin{theorem}[cf. Theorem~2 in~\cite{warren}]\label{thm:warren}
Let $Q_1, \ldots, Q_m$ be polynomials in $n$ variables, each of degree $d$ or less. The number of topological components of the set $\{x \in \mathbb{R}^n|\forall i\in \{1,\ldots,m\}: Q_i(x)\neq 0\}$ of common non-zeroes is at most $$2(2d)^n\sum_{k=0}^{n}{2^k\binom{m}{k}}.$$
\end{theorem}

Using this theorem, we obtain the following upper bound on $t_s(n,2)$. 
\begin{proposition}\label{prop:poly}
For every $n \ge 3$, we have $$t_s(n) \le 2\cdot 16^n\cdot \sum_{k=0}^{2n}{2^k\binom{\binom{n}{3}}{k}}.$$
\end{proposition}
We remark that the proof method is standard and uses the same idea that is employed by Alon in his paper~\cite{alon}. The improvement of the bound solely relies on using Warren's theorem, which provides improved upper bounds compared to corresponding result used by Alon (namely, Theorem~2.4 in~\cite{alon}). 
\begin{proof}[Proof of Proposition~\ref{prop:poly}]
Given a labelled point set $P=\{p_1,\ldots,p_n\}$ in $\mathbb{R}^2$, its \emph{sign pattern} $\mathbf{sgn}(P) \in \{-1,0,1\}^{\binom{n}{3}}$ is the vector indexed by the ordered triples $\{ijk|1\le i<j<k\le n\}$ (in lexicographic order) and whose entry at position $ijk$ represents the orientation of the point triple $p_ip_jp_k$. 

Note that by definition, two labelled $n$-point sets $P,Q$ in the plane are isomorphic if and only if $\mathbf{sgn}(P)=\mathbf{sgn}(Q)$\footnote{It is okay for us to restrict to triples $ijk$ with $i<j<k$ only, since the signs all other triples can be easily deduced from those by changing the sign for every inversion of two indices.}. Moreover, a point set $P$ is in general position if and only if no three points lie on a line, which is equivalent to saying that all entries of its sign pattern $\mathbf{sgn}(P)$ are non-zero (that is, $-1$ or $1$). Altogether, we may conclude that $t_s(n,2)$ is bounded by the size of the set of possible $\{-1,1\}$-sign patterns that can be generated by choosing $n$ points in the plane:
$$t_s(n,2)\le |\{\mathbf{sgn}(\{p_1,\ldots,p_n\})|p_1,\ldots,p_n \in \mathbb{R}^2\}\cap \{-1,1\}^{\binom{n}{3}}|.$$
To bound the size of this set of $\{-1,1\}$-vectors, let us write out the coordinates of the points $p_1,\ldots,p_n$ as $p_i=(x_i,y_i)$ and note that the entry of $\mathbf{sgn}(\{p_1,\ldots,p_n\})$ at any position $ijk$ can be written as the sign of the determinant
$$\det \begin{pmatrix}
1 & 1 & 1\\
x_i & x_j & x_k\\
y_i & y_j & y_k
\end{pmatrix}.$$ The latter can be viewed as a polynomial $Q_{ijk}(x_1,y_1,x_2,y_2,\ldots,x_n,y_n)$ in the $2n$ variables $x_1,y_1,\ldots,x_2,y_2$ of degree $d=2$. Setting $m:=\binom{n}{3}$, we can thus apply Warren's theorem to the $m$ distinct polynomials $Q_{1,2,3},\ldots,Q_{n-2,n-1,n}$. We thus obtain that the number of topologically connected components of the set 
$$R:=\{(x_1,y_1,\ldots,x_n,y_n) \in \mathbb{R}^{2n}|\forall 1 \le i<j<k\le n: Q_{ijk}(x_1,y_1,\ldots,x_n,y_n)\neq 0\}$$ is upper-bounded by
$$2(2d)^{2n}\sum_{k=0}^{2n}{2^k\binom{m}{k}}=2\cdot 16^n\cdot \sum_{k=0}^{2n}{2^k\binom{\binom{n}{3}}{k}}.$$
It remains to argue why the number of components of this set is an upper-bound on the number of sign-patterns. We claim that the point sets corresponding to the vectors in any fixed component of $R$ all have the same sign-pattern. Since every possible $\{-1,1\}$-sign pattern is clearly generated by at least one of the components of $R$, this will imply the claim of the theorem. 

So, suppose towards a contradiction that there exist vectors $\mathbf{x}=(x_1,y_1,\ldots,x_n,y_n)$ and $ \mathbf{x}'=(x_1',y_1',\ldots,x_n',y_n')$ in the same connected component of $R$, but such that the point sets $P=\{(x_1,y_1),\ldots,(x_n,y_n)\}$ and $P'=\{(x_1',y_1'),\ldots,(x_n',y_n')\}$ satisfy $\mathbf{sgn}(P)\neq \mathbf{sgn}(P')$. This means that there exists at least one ordered triple $ijk$ with $1 \le i<j<k\le n$ such that $$\text{sgn}(Q_{ijk}(x_1,y_1,\ldots,x_n,y_n))\neq \text{sgn}(Q_{ijk}(x_1',y_1',\ldots,x_n',y_n')).$$ Since $Q_{ijk}$ is a polynomial and thus a continuous function, the intermediate value theorem implies that when following any topological path in $\mathbb{R}^{2n}$ from $\mathbf{x}$ to $\mathbf{x}'$, $Q_{ijk}$ will at least once attain the value $0$ along this path. However, any zero of $Q_{ijk}$ is by definition not an element of $R$. Hence, there is no path fully contained in $R$ that connects $\mathbf{x}$ and $\mathbf{x}'$, contradicting that by assumption they belong to the same component of $R$. This shows that our above assumption was wrong. Indeed, $n$-point-sets corresponding to the vectors in a fixed component of $R$ all have the same sign pattern. This concludes the proof of the theorem.
\end{proof}
We can now give the proof of Theorem~\ref{thm:main2}.
\begin{proof}[Proof of Theorem~\ref{thm:main2}]
Let $n \in \{107,\ldots,193\}$ be given. Using Theorem~\ref{thm:aux} and Proposition~\ref{prop:poly}, we obtain that
$$\sigma(n)\le \left\lfloor \frac{\log_2(t_s(n,2))-(n-4)-\log_2((n-3)!)}{n-\log_2(16n(n-1)(n-2))}\right\rfloor +2$$
$$\le \left\lfloor \frac{\log_2\left(2\cdot 16^n\cdot \sum_{k=0}^{2n}{2^k\binom{\binom{n}{3}}{k}})\right)-(n-4)-\log_2((n-3)!)}{n-\log_2(16n(n-1)(n-2))}\right\rfloor +2.$$ Using a simple python script, it can be verified that the latter expression equals $30$ for every $n \in \{107,\ldots,193\}$ (and in fact, is bigger than $30$ for any $n$ below or above this interval). This concludes the proof. 
\end{proof}

\section{Explicit construction}\label{sect:explicit}
In this section, we provide an explicit construction of a conflict collection consisting of $n^{\underline{6}}+1$ planar $n$-vertex graphs for every $n \ge 5040$, thereby proving Theorem~\ref{thm:main3}. Let $n \ge 5040$ be given. Let us start by considering the set $$\mathcal{S}_n:=\left\{(n_1,n_2,n_3,n_4,n_5,n_6,n_7,n_8)\in \mathbb{N}\cup \{0\}\bigg\vert\sum_{i=1}^{8}{n_i}=n-6\right\}$$ of ordered partitions of the number $n-6$ into $8$ non-negative integers. We then have $$|\mathcal{S}_n|=\binom{(n-6)+8-1}{8-1}=\binom{n+1}{7}=\frac{(n+1)^{\underline{7}}}{7!}=\frac{(n+1)n^{\underline{6}}}{7!}>n^{\underline{6}},$$ where we used that $n \ge 5040=7!$ in the last line. It is therefore possible to select (explicitly) a subset $\mathcal{S}_n'\subset \mathcal{S}_n$ such that $|\mathcal{S}_n'|=n^{\underline{6}}+1$. We now define a collection $\mathcal{G}_n$ of planar triangulations on $n$ vertices as follows:

Let $H$ denote the octahedron graph, which is a triangulation with $6$ vertices $a,b,c,d,e,f$ and $8$ faces $f_1,\ldots,f_8$ (see Figure~\ref{fig:octa}). 

Now, for every element $(n_1,n_2,n_3,n_4,n_5,n_6,n_7,n_8) \in \mathcal{S}_n'$, construct a planar triangulation by (1) adding for every one of the $8$ faces $f_i$ of $H$ a set of exactly $n_i$ additional vertices that are placed into the corresponding face, and (2) triangulating each of the faces $f_i$ together with the $n_i$ new vertices placed in it. 

Clearly, there may be many different ways to triangulate, but for the sake of this construction, we just choose and fix for every given $(n_1,n_2,n_3,n_4,n_5,n_6,n_7,n_8) \in \mathcal{S}_n'$ one possible way to do this, and call the resulting triangulation $T_n(n_1,n_2,n_3,n_4,n_5,n_6,n_7,n_8)$.

Finally, we define $\mathcal{G}_n=\{T_n(n_1,n_2,n_3,n_4,n_5,n_6,n_7,n_8)|(n_1,n_2,n_3,n_4,n_5,n_6,n_7,n_8)\in \mathcal{S}_n'\}.$

Having defined (explicitly) the family $\mathcal{G}_n$ of planar $n$-vertex graphs, we can now move on to the proof of Theorem~\ref{thm:main3}.

\begin{figure}[h]
    \centering
    \includegraphics[scale=0.65]{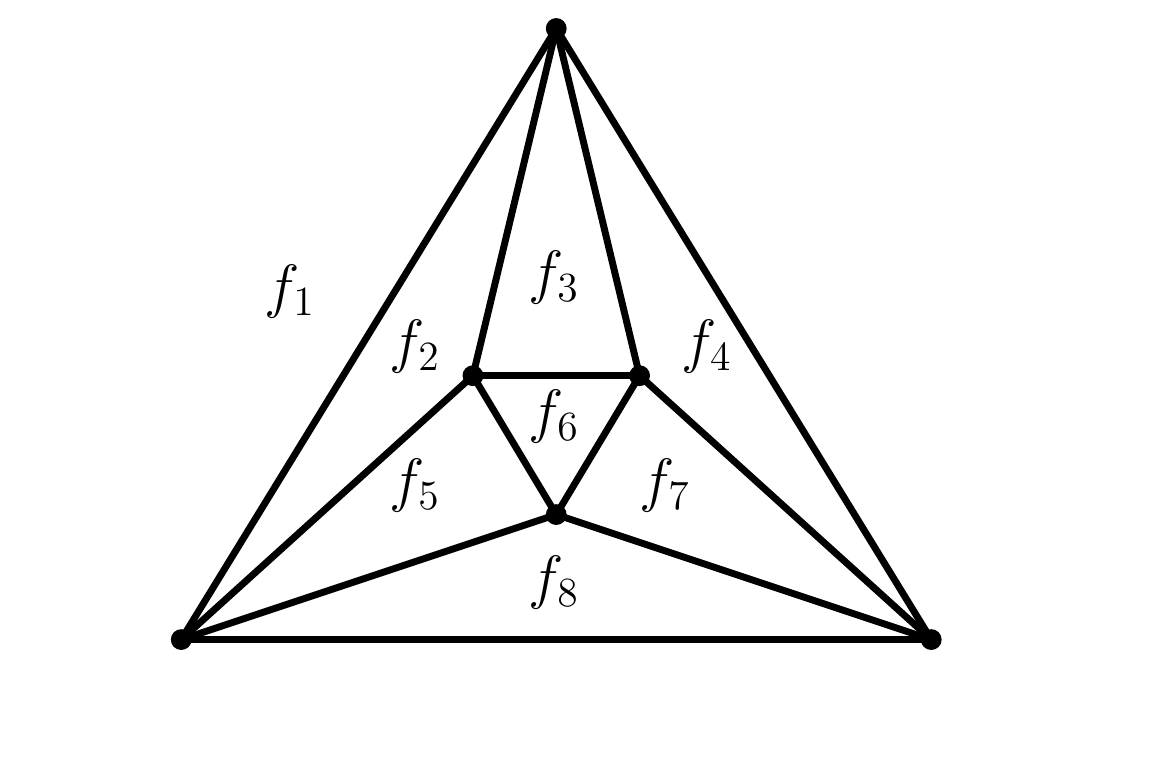}     \includegraphics[scale=0.65]{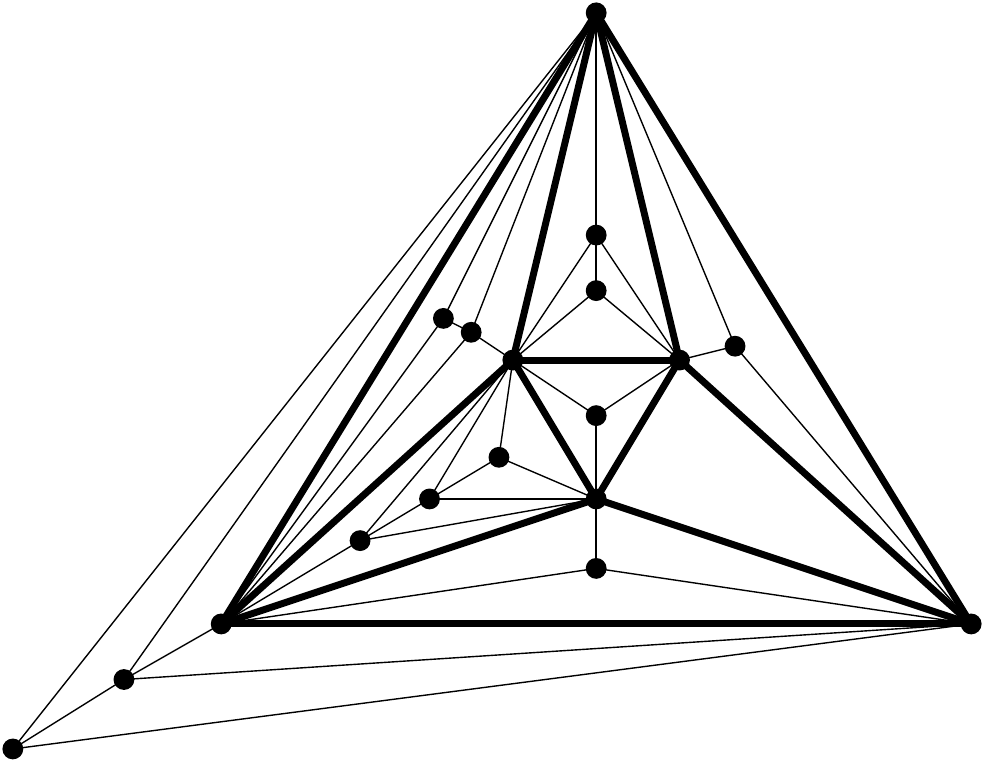}
    \caption{Left: The octahedron with its 8 different faces $f_1,\ldots,f_8$. Right: A possible choice for $T_{18}(2,2,2,1,3,1,0,1)$. The subgraph $H$ is marked fat.}
    \label{fig:octa}
\end{figure}

\begin{proof}[Proof of Theorem~\ref{thm:main3}]
Let $n \ge 5040$ be given and $\mathcal{G}_n$ be defined as above. We claim that $\mathcal{G}_n$ is a conflict collection, i.e., there is no simultaneous embedding of $\mathcal{G}_n$. Towards a contradiction, suppose that there is a set $P$ of $n$ points in the plane such that $\mathcal{G}_n$ is simultaneously embeddable on $P$. For every $s \in \mathcal{S}_n'$ let $\pi_s:T_n(s)\rightarrow P$ denote a straight-line embedding of $T_n(s)$ on $P$. Then for each $s \in \mathcal{S}_n'$, the restriction $\pi_s|_{\{a,b,c,d,e,f\}}$ to the vertex-set of $H$ forms an injective mapping from $\{a,b,c,d,e,f\}$ into $P$. Since the number of such mappings equals $n^{\underline{6}}$ and since $|\mathcal{S}_n'|=n^{\underline{6}}+1$, the pigeon-hole principle implies that there are distinct $s,s' \in \mathcal{S}_n'$ such that $\pi_{s}(x)=\pi_{s'}(x)$ for every $x\in \{a,b,c,d,e,f\}$. This means in particular that every face $f_i$ of $H$ gets embedded as a triangle on $P$ in the same way by $\pi_{s}$ and $\pi_{s'}$. 

Write $s=(n_1,\ldots,n_8)$ and $s'=(n_1',\ldots,n_8')$. Since $s \neq s'$, and $\sum_{i=1}^{8}{n_i}=n-6=\sum_{i=1}^{8}{n_i'}$, there have to be at least two indices $i \in \{1,\ldots,8\}$ such that $n_i\neq n_i'$. We can therefore choose $i \in \{1,\ldots,8\}$ such that $n_i \neq n_i'$ and $f_i$ is not the outer face in the embedding of $H$ induced by $\pi_s|_{\{a,b,c,d,e,f\}}=\pi_{s'}|_{\{a,b,c,d,e,f\}}$.

Now note that since $T_n(s)$, $T_n(s')$ are planar triangulations (and thus $3$-connected planar graphs), Whitney's theorem implies that they have combinatorially unique embeddings into the plane, apart from the choice of the outer face. This implies in particular that in the embeddings $\pi_{s}$ and $\pi_{s'}$ of $T(s)$ and $T(s')$, there are exactly $n_i$ vertices of $T_n(s)$ inside the triangle spanned by $f_i$, but on the other hand also exactly $n_i'$ vertices of $T_n(s')$ inside the very same triangle. Since every point of $P$ must be used by both embeddings, it follows that $n_i=n_i'$, a contradiction. This shows that our initial assumption on the simultaneous embeddability of $\mathcal{G}_n$ was wrong, and concludes the proof of Theorem~\ref{thm:main3}.
\end{proof}

\paragraph{\textbf{Declaration.}} The authors declare that no data has been used for the research presented in this article.

\end{document}